\documentclass{amsart}

\usepackage{lineno}
\usepackage{amsfonts}
\usepackage{tipa}
\usepackage{amssymb}
\usepackage{mathrsfs}
\usepackage{amsmath}
\usepackage{txfonts}
\newtheorem{theorem}{Theorem}

\newtheorem{lemma}{Lemma}
\newtheorem{claim}{Claim}
\newtheorem{step}{Step}
\newtheorem{case}{Case}
\newtheorem{subcase}{Case}[case]

\newtheorem{remark}{Remark}

\providecommand{\newproof}[2]{}

\usepackage[colorlinks=true]{hyperref}
\IfFileExists{BOONDOX-cal.sty}{\usepackage{BOONDOX-cal}}{\IfFileExists{boondox-cal.sty}{\usepackage{boondox-cal}}{}}
\hypersetup{urlcolor=red, citecolor=blue}

\begin{document}
\title[Symmetry of solutions to nonlocal Neumann problems]{Symmetry of solutions to nonlocal Neumann problems with the fractional Laplacian in half spaces}

\author{Baiyu Liu}
\address{School of Mathematics and Physics, University of Science and Technology Beijing, 30 Xueyuan Road, Haidian District, Beijing, 100083, P.R. China}
\email{liubymath@gmail.com; liuby@ustb.edu.cn}
\thanks{Corresponding author. Tel.: +86 13466547073. Project supported by the National Natural Science Foundation of China (No. 12471089).}

\author{Wenli Tan}
\address{School of Mathematics and Physics, University of Science and Technology Beijing, 30 Xueyuan Road, Haidian District, Beijing, 100083, P.R. China}


\begin{abstract}
This paper investigates the fractional Laplacian equation in half-spaces, focusing on cases with nonlocal Neumann boundary conditions. We consider both single equations and systems, and establish two key principles: the decay at infinity principle and the narrow region principle. These foundational results, combined with the direct method of moving planes, allow us to derive symmetry and monotonicity properties for decaying solutions. To the best of our knowledge, these are the first results concerning the symmetry of solutions to fractional Laplacian equations with nonlocal Neumann conditions.
\end{abstract}

\maketitle

\section{Introduction}
In this paper, we are concerned with the following semilinear nonlocal Neumann problem 
\begin{linenomath*}\begin{equation}
	\left\{
	\begin{array}{lc} \label{single}
		(-\Delta)^\frac{\alpha}{2}u(x)=f\big(u\big), & x\in \mathbb{R}^n_+, \\
		\mathcal{N}_{\frac{\alpha}{2}}u(x)=g\big(u\big),& x\in \mathbb{R}^n\setminus\overline{\mathbb{R}^n_+},\\
	\end{array}
	\right.
\end{equation}\end{linenomath*}
where $\mathbb{R}^n_+=\{x=(x_1,\cdots,x_n)\mid x_n>0\}$.
Here, $(-\Delta)^\frac{\alpha}{2}$ is the fractional Laplacian defined as
\begin{linenomath*}\begin{equation}
\label{definition}
(-\Delta)^\frac{\alpha}{2}u(x)=C_{n,\alpha}P.V.\int_{\mathbb{R}^n}\frac{u(x)-u(y)}{|x-y|^{n+{\alpha}}}dy=C_{n,\alpha}\lim_{\epsilon \rightarrow 0}\int_{\mathbb{R}^n\backslash B_\epsilon(x)}\frac{u(x)-u(y)}{|x-y|^{n+{\alpha}}}dy,
\end{equation}\end{linenomath*}
where $\alpha\in (0,2)$, $C_{n,\alpha}$ is some normalization constant.
$\mathcal{N}_{\frac{\alpha}{2}}$ is the nonlocal normal derivative \cite{DRV}, defined as
\begin{linenomath*}
$$\mathcal{N}_{\frac{\alpha}{2}}u(x):=C_{n,\alpha}\int_{\mathbb{R}_+^n}\frac{u(x)-u(y)}{|x-y|^{n+{\alpha}}}dy,\quad \textrm{for each}\ x\in\mathbb{R}^n\setminus\overline{\mathbb{R}^n_+}.
$$ 
\end{linenomath*}

The nonlocal normal derivative $\mathcal{N}_{\frac{\alpha}{2}}$, introduced by Dipierro, Ros-Oton and Valdinoci \cite{DRV}, is a natural nonlocal analogue of the classical outward Neumann derivative. A key observation is that, unlike the fractional Laplacian $(-\Delta)^{\frac{\alpha}{2}}$ which is defined via a principal value integral with built-in cancellation, the integral defining $\mathcal{N}_{\frac{\alpha}{2}}u(x)$ lacks such cancellation. Therefore $\mathcal{N}_{\frac{\alpha}{2}}u(x)$ is well defined only when $x\in\mathbb{R}^n\setminus\overline{\mathbb{R}^n_+}$ (i.e., $x_n<0$), where every $y\in\mathbb{R}^n_+$ satisfies $|x-y|>0$. This is why the exterior conditions in \eqref{single} and \eqref{system} are imposed on $\mathbb{R}^n\setminus\overline{\mathbb{R}^n_+}$ rather than on all of $\mathbb{R}^n\setminus\mathbb{R}^n_+$. The operator satisfies the following integration by parts formulas for bounded $C^2$ functions $u$, $v$ \cite{DRV}:
\begin{linenomath*}$$
\int_{\Omega}(-\Delta)^\frac{\alpha}{2}u\,dx=-\int_{\mathbb{R}^{n}\backslash \Omega}\mathcal{N}_{\frac{\alpha}{2}}u\,dx,
$$ \end{linenomath*}
and
\begin{linenomath*}$$
\frac{C_{n,\alpha}}{2}\int_{\mathbb{R}^{2n}\backslash(\mathcal{C}\Omega)^2}\frac{(u(x)-u(y))(v(x)-v(y))}{|x-y|^{n+{\alpha}}}dx\,dy=\int_{\Omega}v(-\Delta)^\frac{\alpha}{2}u+\int_{\mathbb{R}^{n}\backslash \Omega}v\,\mathcal{N}_{\frac{\alpha}{2}}u.
$$ \end{linenomath*}
Furthermore, as $\frac{\alpha}{2}\to 1$, $\mathcal{N}_{\frac{\alpha}{2}}u$ converges to the classical outward normal derivative $\partial_\nu u$ \cite{DRV}, confirming that problem \eqref{single} is a genuine generalization of the classical semilinear Neumann problem to the fractional setting.

From a stochastic viewpoint, the coupling of $(-\Delta)^{\frac{\alpha}{2}}$ and $\mathcal{N}_{\frac{\alpha}{2}}$ describes the infinitesimal generator of a censored $\alpha$-stable process in $\mathbb{R}^n_+$: a L\'{e}vy process that is not killed upon leaving the domain, but instead interacts with the exterior region $\mathbb{R}^n\setminus\overline{\mathbb{R}^n_+}$ through the kernel $|x-y|^{-(n+\alpha)}$ \cite{DRV}. The nonlinear boundary datum $g(u)$ (respectively $h, l$ in the system) models the reaction intensity at which particles in the exterior feed back into the dynamics. This framework arises naturally in the study of anomalous diffusion \cite{GWN}
and nonlocal population dynamics with cross-boundary interactions \cite{DLVL}.

The mathematical analysis of problems with nonlocal Neumann conditions has developed actively in recent years. Abatangelo \cite{N} showed that the fractional Laplacian of functions satisfying homogeneous nonlocal Neumann conditions can be expressed as a regional operator with a logarithmic kernel near the boundary. Cinti and Colasuonno \cite{CC} established the existence of nonnegative radial $C^2$ solutions to the problem
\begin{linenomath*}\begin{equation*}
	\left\{
	\begin{array}{ll}
		(-\Delta)^\frac{\alpha}{2}u(x)+u=f(u), & x\in \Omega, \\
		\mathcal{N}_{\frac{\alpha}{2}}u(x)=0,& x\in \mathbb{R}^n\backslash\overline{\Omega},\\
	\end{array}
	\right.
\end{equation*}\end{linenomath*}
where $\alpha/2>\frac{1}{2}$, $\Omega$ is a ball or annulus, and $f$ is superlinear. Audrito, Felipe-Navarro and Ros-Oton \cite{ANR} proved optimal boundary regularity: any weak solution of $(-\Delta)^{\frac{\alpha}{2}}u=f$ in $\Omega$ with $\mathcal{N}_{\frac{\alpha}{2}}u=0$ in $\mathbb{R}^n\setminus\overline{\Omega}$ belongs to $C^\alpha$ up to the boundary. For further results on trace spaces and well-posedness, we refer to \cite{GHR}. Despite these advances, the qualitative properties of solutions --- in particular symmetry and monotonicity --- had not previously been studied for fractional Laplacian equations with nonlocal Neumann conditions.

The method of moving planes is a fundamental tool for establishing radial and axial symmetry of solutions to elliptic equations. Originating with Alexandrov in the 1950s, it was further developed by Serrin \cite{33}, Gidas, Ni and Nirenberg \cite{GNN,17}, Berestycki and Nirenberg \cite{BN}, Berestycki, Caffarelli and Nirenberg \cite{Berestycki97}, Caffarelli, Gidas and Spruck \cite{3}, Chen and Li \cite{CLA,4}, Li and Zhu \cite{LZ}, Lin \cite{20}, and Chen, Li and Ou \cite{7}, among others. For fractional Laplacian equations with Dirichlet-type conditions on half spaces, Chen, Li and Li \cite{CLL} established the direct method of moving planes in the nonlocal setting; this has since been applied widely to derive symmetry and monotonicity of solutions, see \cite{DQW, GP, GP2, LM, WN} and the references therein.

The nonlocal Neumann setting introduces difficulties that are absent in the Dirichlet case. In the Dirichlet setting, the solution vanishes outside $\mathbb{R}^n_+$, so the reflected function satisfies a straightforward comparison. Under the nonlocal Neumann condition $\mathcal{N}_{\frac{\alpha}{2}}u=g(u)$, the exterior values of $u$ depend nonlinearly on the solution itself. Establishing the key positivity condition $\mathcal{N}_{\frac{\alpha}{2}}w_\lambda\geq 0$ for the difference $w_\lambda=u_\lambda-u$ requires the monotonicity of $g$ together with a careful analysis of points near the boundary $\{x_n=0\}$ (where $\mathcal{N}_{\frac{\alpha}{2}}$ is not defined). To complete the moving planes procedure, new maximum principles adapted to this mixed structure are required. In this paper, we develop a \emph{decay at infinity principle} and a \emph{narrow region principle} for the nonlocal Neumann setting, and apply them to establish the symmetry of decaying solutions to \eqref{single} and \eqref{system}. To the best of our knowledge, Theorems~\ref{th1} and~\ref{fczth1} are the first symmetry results for fractional Laplacian equations with nonlocal Neumann conditions.

We require $u\in L_\alpha\cap C_{loc}^{1,\theta}(\mathbb{R}^n)$ for some $\theta>\max\{0,\alpha-1\}$, where
\begin{linenomath*}$$L_\alpha=\left\{u:\mathbb{R}^n \rightarrow \mathbb{R} \mid \int_{\mathbb{R}^n}\frac{|u(x)|}{1+|x|^{n+{\alpha}}}dx<\infty\right\}.$$ \end{linenomath*}
The space $L_\alpha$ controls the growth at infinity, ensuring convergence of the integral in \eqref{definition} at infinity.
The H\"{o}lder condition $C^{1,\theta}$ with $\theta>\alpha-1$ ensures that the local singularity of the integrand near $x$ is integrable (which is why the principal value prescription is needed): expanding $u(y)=u(x)+\nabla u(x)\cdot(y-x)+O(|y-x|^{1+\theta})$, the remainder contributes $\int_0^{\varepsilon}r^{\theta-\alpha}dr$, which is finite if and only if $\theta>\alpha-1$. When $\alpha\in(0,1]$ any $\theta>0$ suffices; when $\alpha\in(1,2)$ we need $\theta\in(\alpha-1,1)$.

Our main results are the following.

\begin{theorem} \label{th1}
Let $u(x) \in L_\alpha\cap C_{loc}^{1,\theta}(\mathbb{R}^n)$, for some $\theta>\max\{0,\alpha-1\}$, be a non-negative solution of problem \eqref{single} satisfying 
	\begin{linenomath*}\begin{equation}
		\label{udecay}
		u(x)=o\left(\frac{1}{|x|^\beta}\right),\text{ as }|x|\rightarrow\infty,
	\end{equation}\end{linenomath*}
where $\beta$ is a positive constant.
Assume 
	\begin{linenomath*}\begin{equation}
		\label{fdecay}
		f'(t)\leq t^{\gamma},  \textrm{as}\ t\to 0^+ ,\quad  \beta\gamma\geq\alpha.
	\end{equation}\end{linenomath*}
Assume furthermore that $g$ is non-increasing on $[0,d)$ for some $d>0$. Then $u(x)$ is axially symmetric about a line parallel to the $x_n$-axis, i.e.,
$u(x)=u(|\bar x-\bar x^0|,x_n)$, for some $\bar x^0\in \mathbb{R}^{n-1}$.
\end{theorem}

\begin{remark}
A typical example satisfying the hypotheses of Theorem~\ref{th1} is
$f(u)=u^p$ ($p>1$) and $g(u)=-u^q$ ($q\geq 1$).
Indeed, $f'(t)=pt^{p-1}$ satisfies \eqref{fdecay} with $\gamma=p-1$ whenever $\beta(p-1)\geq\alpha$, and $g(u)=-u^q$ is non-increasing on $[0,+\infty)$.
\end{remark}

We also consider the following nonlocal Neumann fractional Laplacian system on half spaces
\begin{linenomath*}\begin{equation}
	\label{system}
	\left\{
	\begin{array}{ll}
		(-\Delta)^\frac{\alpha}{2}u(x)=f(u,v), & x\in \mathbb{R}^n_+, \\
		(-\Delta)^\frac{\alpha}{2}v(x)=g(u,v),& x \in \mathbb{R}^n_+,\\
		\mathcal{N}_{\frac{\alpha}{2}}u(x)=h(u,v),& x\in \mathbb{R}^n\setminus\overline{\mathbb{R}^n_+},\\
		\mathcal{N}_{\frac{\alpha}{2}}v(x)=l(u,v),& x\in \mathbb{R}^n\setminus\overline{\mathbb{R}^n_+},\\
	\end{array}
	\right.
\end{equation}\end{linenomath*}
and obtain symmetry results for its positive solutions.

\begin{theorem}\label{fczth1}
Let $(u,v) \in \left(L_\alpha\cap C_{loc}^{1,\theta}(\mathbb{R}^n)\right)^2$, for some $\theta>\max\{0,\alpha-1\}$, be a positive solution of system \eqref{system} with $f,g,h,l\in C^1([0,+\infty)\times[0,+\infty),\mathbb{R})$. We suppose that
\begin{linenomath*}
	\begin{eqnarray*}
&(i)&\quad u(x)\leq \frac{1}{|x|^a},\quad v(x)\leq \frac{1}{|x|^b},\quad \text{as }|x|\to\infty;\\
&(ii)&\quad \frac{\partial f}{\partial u}\leq u^{p-1}v^q,\quad  \frac{\partial g}{\partial v}\leq u^rv^{s-1},\quad \text{as }(u,v)\to(0^+,0^+);\\
&(iii)&\quad \frac{\partial f}{\partial v}\leq u^pv^{q-1},\quad \frac{\partial g}{\partial u}\leq u^{r-1}v^s,\quad \text{as }(u,v)\to(0^+,0^+);\\
&(iv)&\quad \frac{\partial f}{\partial v}>0,\quad \frac{\partial g}{\partial u}>0,\quad\forall(u,v)\in \mathbb{R}^+\times\mathbb{R}^+,\\
&(v)&\quad\frac{\partial h}{\partial u}<0,\quad\frac{\partial l}{\partial v}<0,\quad \forall(u,v)\in \mathbb{R}^+\times\mathbb{R}^+,\\
& &\quad\frac{\partial h}{\partial v}>0,\quad\frac{\partial l}{\partial u}>0,\quad \forall(u,v)\in \mathbb{R}^+\times\mathbb{R}^+,\\
&(vi)&\quad \text{there exist }\mu,\nu>0\text{ with }\mu\nu>1\text{ such that}\\
& &\quad -\frac{\partial h}{\partial u}\geq \mu\frac{\partial h}{\partial v},\quad -\frac{\partial l}{\partial v}\geq \nu\frac{\partial l}{\partial u},\quad\forall(u,v)\in \mathbb{R}^+\times\mathbb{R}^+.
\end{eqnarray*}
\end{linenomath*}
where $a,b>0;p,q,r,s\geq1$. If $\alpha\in(0,2)$ satisfies
\begin{linenomath*}$$(vii)\quad \alpha<\min\{ap+bq-a,ap+bq-b,ar+bs-a,ar+bs-b\},$$ \end{linenomath*}
then $u(x)$ and $v(x)$ are axially symmetric about a line parallel to the $x_n$-axis, i.e.,
$u(x)=u(|\bar x-\bar x^0|,x_n)$ and $v(x)=v(|\bar x-\bar x^0|,x_n)$, for some $\bar x^0\in \mathbb{R}^{n-1}$.
\end{theorem}
\begin{remark}
A model example for Theorem~\ref{fczth1} is $f(u,v)=u^p v^q$ and $g(u,v)=u^r v^s$ with $p,q,r,s\geq 1$, and
$$h(u,v)=-Mu+Kv,\quad l(u,v)=Ku-Nv,$$
with constants $M,N>0$ and $0<K<\sqrt{MN}$. Then conditions (i)--(iv) are verified directly, while (v)--(vi) hold with $\mu=M/K$ and $\nu=N/K$.
Condition (vii) is satisfied for all sufficiently small $\alpha>0$, depending on the decay rates $a,b$ and the exponents $p,q,r,s$.
\end{remark}

The paper is organized as follows. Section \ref{sec:eq} is devoted to the scalar equation \eqref{single}, encompassing the essential ingredients required in the method of moving planes for a single fractional Laplacian problem with nonlocal Neumann condition. This includes the decay at infinity principle (Lemma~\ref{Decay}), the narrow region principle (Lemma~\ref{Narrow}), and the proof of Theorem~\ref{th1}. The proof of Theorem~\ref{fczth1} is presented in Section \ref{sec:sys}. Throughout the paper, $c$ and $C$ will denote positive constants that may vary from line to line.

\section{Symmetry results for nonlocal Neumann fractional Laplacian equations}\label{sec:eq}
In what follows, we apply the direct method of moving planes. 
Recall the upper half space 
$\mathbb{R}^n_+=\mathbb{R}^{n-1}\times \mathbb{R}_+=\{x=(x_1,\cdots,x_n)\mid x_n>0\}$. 
Choose any direction in $\mathbb{R}^{n-1}$ to be the $x_1$ direction. For each $\lambda \in \mathbb{R}$, we write $x=(x_1,x')$ with $x'=(x_2,\cdots,x_n)\in \mathbb{R}^{n-1}$ and define
\begin{linenomath*}
$$
	T_\lambda=\{x\in \mathbb{R}^n \mid x_1=\lambda\}, \quad 
	H_\lambda=\{x\in \mathbb{R}^n \mid x_1<\lambda\},\quad 
	\Sigma_\lambda=\{x \in \mathbb{R}_+^n \mid x_1<\lambda \}.
$$
\end{linenomath*}	
For each point $x=(x_1,x')\in H_\lambda$ and let $x^\lambda=(2\lambda-x_1,x')$ be the reflected point with respect to the hyperplane $T_\lambda$. Define the reflected functions by $u_\lambda(x)=u(x^\lambda)$ and let
\begin{linenomath*}\begin{equation}\label{def:w}
	w_\lambda(x)=u(x^\lambda)-u(x).
\end{equation}\end{linenomath*}
For the sake of clarity in exposition, we introduce notations
\begin{linenomath*}
$$
	\tilde{H}_\lambda=\{x\in \mathbb{R}^n\mid x^\lambda \in H_\lambda\},\quad 
	\tilde{\Sigma}_\lambda=\{x\in \mathbb{R}^n\mid x^\lambda \in \Sigma_\lambda\},\quad
	\Sigma_\lambda^-=\{x\in \Sigma_\lambda \mid w_\lambda(x)<0\}.
$$
\end{linenomath*}	

Now we will introduce the following two principles, which will play important roles in the proof of Theorem \ref{th1}.

\begin{lemma}[Decay at Infinity]\label{Decay}
Let $\Omega_1\subset\{x\in H_\lambda \mid x_n\geq 0\}$ and $\Omega_2 \subset \{x\in H_\lambda \mid x_n< 0\}$. Assume that $w(x)\in L_\alpha\cap C_{loc}^{1,\theta}(\mathbb{R}^n)$, for some $\theta>\max\{0,\alpha-1\}$, satisfies
	\begin{linenomath*}\begin{equation}
	\label{lemma1}
		\left\{
		\begin{array}{ll}
			(-\Delta)^\frac{\alpha}{2}w(x)+a(x)w(x)\geq 0, & x\in \Omega_1, \\
			\mathcal{N}_{\frac{\alpha}{2}}w(x)\geq 0,& x\in \Omega_2,\\
			w(x)\geq 0,& x \in H_\lambda\backslash(\Omega_1 \cup \Omega_2),\\
			w(x^{\lambda})=-w(x),& x \in H_\lambda,
		\end{array}
		\right.
	\end{equation}\end{linenomath*}
	with 
	\begin{linenomath*}\begin{equation}
		\label{condition of c}
		\varliminf_{|x|\to\infty,x\in \Omega_1 } |x|^{\alpha}a(x)\geq 0.
	\end{equation}\end{linenomath*}
Then there exists a constant $R_0>0$, such that for all $\lambda <0$, if
	\begin{linenomath*}\begin{equation}
		\label{minu}
		w({x_0})=\min_{\overline{\Omega_1}}w(x)<0,
	\end{equation}\end{linenomath*}
	then
	\begin{linenomath*}\begin{equation}
		\label{|x0|}
		|{x_0}|\leq R_0.
	\end{equation}\end{linenomath*}
\end{lemma}
\begin{remark}
Here $R_0$ depends on $a(x)$, but is independent of both $w(x)$ and $\lambda$. In the application to Theorem \ref{th1}, the coefficient $a(x)=-f'(\psi(\lambda,x))$ satisfies $a(x)\geq -u(x)^{\gamma}$, and the lower bound $-u(x)^{\gamma}$ depends only on $u(x)$ through conditions \eqref{udecay}--\eqref{fdecay}, uniformly for all $\lambda\leq 0$. Therefore $R_0$ can be chosen independently of $\lambda$.
\end{remark}
\begin{proof}
By the definition of $(-\Delta)^\frac{\alpha}{2}$, we have
\begin{linenomath*}\begin{equation}\label{compute1}
	(-\Delta)^\frac{\alpha}{2}w({x_0}) =C_{n,\alpha}P.V.\int_{\mathbb{R}^n}\frac{w({x_0})-w(y)}{|{x_0}-y|^{n+{\alpha}}}dy =I_1+I_2,
\end{equation}\end{linenomath*}
where 
\begin{linenomath*}
$$
I_1=C_{n,\alpha}P.V.\int_{\mathbb{R}_+^n}\frac{w({x_0})-w(y)}{|{x_0}-y|^{n+{\alpha}}}dy, \quad I_2=C_{n,\alpha}\int_{\mathbb{R}_-^n}\frac{w({x_0})-w(y)}{|{x_0}-y|^{n+{\alpha}}}dy.
$$ 
\end{linenomath*}
Since $w(x)$ is non-negative in $\{x\in H_\lambda \mid x_n\geq 0\}\backslash \Omega_1$, if $x_0$ is the negative minimum point of $w$ in $\Omega_1$, then $w(x_0)\leq w(y)$ for every $y\in \Sigma_\lambda$. Thus, by a direct computation, we arrive at
	\begin{linenomath*}\begin{align}\notag
		I_1 &=C_{n,\alpha}P.V.\int_{\mathbb{R}_+^n}\frac{w({x_0})-w(y)}{|{x_0}-y|^{n+{\alpha}}}dy\\
		&=C_{n,\alpha}P.V.\bigg\{ \int_{\Sigma_\lambda}\frac{w({x_0})-w(y)}{|{x_0}-y|^{n+{\alpha}}}dy+\int_{\tilde{\Sigma}_\lambda}\frac{w({x_0})-w(y)}{|{x_0}-y|^{n+{\alpha}}}dy\bigg\}\notag\\
		&=C_{n,\alpha}P.V.\bigg\{ \int_{\Sigma_\lambda}\frac{w({x_0})-w(y)}{|{x_0}-y|^{n+{\alpha}}}dy+\int_{\Sigma_\lambda}\frac{w({x_0})-w(y^\lambda)}{|{x_0}-y^\lambda|^{n+{\alpha}}}dy\bigg\}\notag\\
		&=C_{n,\alpha}P.V.\bigg\{ \int_{\Sigma_\lambda}\frac{w({x_0})-w(y)}{|{x_0}-y|^{n+{\alpha}}}dy+\int_{\Sigma_\lambda}\frac{w({x_0})+w(y)}{|{x_0}-y^\lambda|^{n+{\alpha}}}dy\bigg\}\notag\\
		&\leq C_{n,\alpha} \int_{\Sigma_\lambda}\bigg\{\frac{w({x_0})-w(y)}{|{x_0}-y^\lambda|^{n+{\alpha}}}+\frac{w({x_0})+w(y)}{|{x_0}-y^\lambda|^{n+{\alpha}}}\bigg\}dy\notag\\
		&=C_{n,\alpha}\int_{\Sigma_\lambda}\frac{2w({x_0})}{|{x_0}-y^\lambda|^{n+{\alpha}}}dy.\label{compute2}
	\end{align}\end{linenomath*}

To estimate $I_2$, we claim that 
\begin{linenomath*}\begin{equation}\label{eq:lem}
w(x)>w({x_0})=\min_{\overline{\Omega_1}}w(x),\quad \forall x \in H_\lambda\cap\{x_n<0\}.
\end{equation}\end{linenomath*}
That is to say, the negative minimum of $w$ on $\overline{\Omega_1}$ is also a lower bound for $w$ on $H_\lambda\cap\{x_n<0\}$.

Suppose otherwise, there exists some $\overline{x}\in H_\lambda\cap\{x_n<0\}$ such that
\begin{linenomath*}$$
w(\overline{x})\leq w({x_0})<0.
$$ \end{linenomath*}
Since $\overline{x}_n<0$, we have $\overline{x}\in \Omega_2$.
Then by \eqref{lemma1}, we have
\begin{linenomath*}\begin{equation}
	\label{Ns>0}
	\mathcal{N}_{\frac{\alpha}{2}}w(\overline{x})\geq 0.
\end{equation}\end{linenomath*}
It follows from a similar calculations as in \eqref{compute2} that
\begin{linenomath*}$$\mathcal{N}_{\frac{\alpha}{2}}w(\overline{x})=C_{n,\alpha}\int_{\mathbb{R}_+^n}\frac{w(\overline{x})-w(y)}{|\overline{x}-y|^{n+{\alpha}}}dy\leq C_{n,\alpha}\int_{\Sigma_\lambda}\frac{2w(\overline{x})}{|\overline{x}-y^\lambda|^{n+{\alpha}}}dy<0.$$ \end{linenomath*}
This contradicts \eqref{Ns>0}. 

Now we proceed to estimate $I_2$.
By using \eqref{eq:lem}, we derive
	\begin{linenomath*}\begin{align}\notag
	I_2 &=C_{n,\alpha}\int_{\mathbb{R}_-^n}\frac{w({x_0})-w(y)}{|{x_0}-y|^{n+{\alpha}}}dy\\
	&=C_{n,\alpha}\bigg\{ \int_{H_\lambda\backslash\Sigma_\lambda}\frac{w({x_0})-w(y)}{|{x_0}-y|^{n+{\alpha}}}dy+\int_{\tilde{H}_\lambda\backslash \tilde{\Sigma}_\lambda}\frac{w({x_0})-w(y)}{|{x_0}-y|^{n+{\alpha}}}dy\bigg\}\notag\\
	&=C_{n,\alpha}\bigg\{ \int_{H_\lambda\backslash\Sigma_\lambda}\frac{w({x_0})-w(y)}{|{x_0}-y|^{n+{\alpha}}}dy+\int_{H_\lambda\backslash\Sigma_\lambda}\frac{w({x_0})-w(y^\lambda)}{|{x_0}-y^\lambda|^{n+{\alpha}}}dy\bigg\}\notag\\
	&=C_{n,\alpha}\bigg\{ \int_{H_\lambda\backslash\Sigma_\lambda}\frac{w({x_0})-w(y)}{|{x_0}-y|^{n+{\alpha}}}dy+\int_{H_\lambda\backslash\Sigma_\lambda}\frac{w({x_0})+w(y)}{|{x_0}-y^\lambda|^{n+{\alpha}}}dy\bigg\}\notag\\
	&\leq C_{n,\alpha} \int_{H_\lambda\backslash\Sigma_\lambda}\bigg\{\frac{w({x_0})-w(y)}{|{x_0}-y^\lambda|^{n+{\alpha}}}+\frac{w({x_0})+w(y)}{|{x_0}-y^\lambda|^{n+{\alpha}}}\bigg\}dy\notag\\
	&=C_{n,\alpha}\int_{H_\lambda\backslash\Sigma_\lambda}\frac{2w({x_0})}{|{x_0}-y^\lambda|^{n+{\alpha}}}dy. \label{compute3}
	\end{align}\end{linenomath*}

Plugging \eqref{compute2} and \eqref{compute3} into \eqref{compute1}, we obtain
	\begin{linenomath*}\begin{align}\notag 
	(-\Delta)^\frac{\alpha}{2}w({x_0})
	&\leq C_{n,\alpha}\int_{\Sigma_\lambda}\frac{2w({x_0})}{|{x_0}-y^\lambda|^{n+{\alpha}}}dy+C_{n,\alpha}\int_{H_\lambda\backslash\Sigma_\lambda}\frac{2w({x_0})}{|{x_0}-y^\lambda|^{n+{\alpha}}}dy\\
	&=C_{n,\alpha}\int_{H_\lambda}\frac{2w({x_0})}{|{x_0}-y^\lambda|^{n+{\alpha}}}dy. \label{compute4}
	\end{align}\end{linenomath*}

For each fixed $\lambda<0$, when $|{x_0}|\geq \lambda$, we have $B_{|{x_0}|}(x^1)\subset \tilde{H}_\lambda$ with $x^1=(3|{x_0}|+({x_0})_1,({x_0})')$, and it follows that
	\begin{linenomath*}\begin{align}
	\int_{H_\lambda}\frac{1}{|{x_0}-y^\lambda|^{n+{\alpha}}}dy
	&=\int_{\tilde{H}_\lambda}\frac{1}{|{x_0}-y|^{n+{\alpha}}}dy \notag\\
	&\geq \int_{B_{|{x_0}|}(x^1)}\frac{1}{|{x_0}-y|^{n+{\alpha}}}dy \notag\\
	&\geq \int_{B_{|{x_0}|}(x^1)}\frac{1}{4^{n+{\alpha}}|{x_0}|^{n+{\alpha}}}dy \notag\\
	&=\frac{\omega_n}{4^{n+{\alpha}}|{x_0}|^{\alpha}},\label{sec1jifen}
	\end{align}\end{linenomath*}
where $\omega_{n}=|B_1(0)|$ in $\mathbb{R}^n$.
	
Then we have
	\begin{linenomath*}\begin{align*}
	0 \leq (-\Delta)^\frac{\alpha}{2}w({x_0})+a({x_0})w({x_0})
	\leq w({x_0})\left(\frac{2\omega_nC_{n,\alpha}}{4^{n+{\alpha}}|{x_0}|^{\alpha}}+a({x_0})\right).
	\end{align*}\end{linenomath*}
That is
	\begin{linenomath*}$$\frac{2\omega_nC_{n,\alpha}}{4^{n+{\alpha}}|{x_0}|^{\alpha}}+a({x_0})\leq 0.$$ \end{linenomath*}
Or equivalently,
	\begin{linenomath*}$$a({x_0})|{x_0}|^{\alpha}\leq -\frac{2\omega_nC_{n,\alpha}}{4^{n+{\alpha}}}<0.$$ \end{linenomath*}
When $|{x_0}|$ is sufficiently large, this contradicts 
\eqref{condition of c}. This completes the proof.
\end{proof}

\begin{lemma}[Narrow Region Principle]\label{Narrow}
Let $\Omega_1$ be a bounded narrow region in $\{x\in H_\lambda \mid x_n\geq 0\}$, such that it is contained in $\{x \in \mathbb{R}^n \mid x_n\geq 0,\ \lambda-l<x_1<\lambda\}$
with small $l$ and $\Omega_2 \subset \{x\in H_\lambda \mid x_n< 0\}$. Assume that $w(x)\in L_\alpha\cap C_{loc}^{1,\theta}(\mathbb{R}^n)$, for some $\theta>\max\{0,\alpha-1\}$. If $a(x)$ is bounded from below in $\Omega_1$ and
	\begin{linenomath*}\begin{equation}
	\label{lemma2}
		\left\{
		\begin{array}{ll}
			(-\Delta)^\frac{\alpha}{2}w(x)+a(x)w(x)\geq 0, & x\in \Omega_1, \\
			\mathcal{N}_{\frac{\alpha}{2}}w(x)\geq 0,& x\in \Omega_2,\\
			w(x)\geq 0,& x \in H_\lambda\backslash(\Omega_1 \cup \Omega_2),\\
			w(x^{\lambda})=-w(x),& x \in H_\lambda.\\
		\end{array}
		\right.
	\end{equation}\end{linenomath*}
Then for sufficiently small $l$, we have
	\begin{linenomath*}\begin{equation}
		\label{Narrow conclusion}
		w(x)\geq 0, \quad x\in \Omega_1.
	\end{equation}\end{linenomath*}
\end{lemma}

\begin{proof}
If \eqref{Narrow conclusion} does not hold, then the continuity of $w$ and \eqref{lemma2} indicate that there exists some point ${x_0} \in \overline{\Omega_1}$ such that
\begin{linenomath*}\begin{equation}
	w({x_0})=\min_{x\in \overline{\Omega}_1}w(x)<0.
\end{equation}\end{linenomath*}

Since $w(x)$ is non-negative in $H_\lambda\backslash (\Omega_1 \cup \Omega_2)$, one can further deduce that ${x_0}\in \Omega_1$ or ${x_0}\in \Omega_2$. 

If ${x_0}\in \Omega_2$, \eqref{lemma2} tells us that
\begin{linenomath*}\begin{equation}
	\label{Nx0>0}
	\mathcal{N}_{\frac{\alpha}{2}}w({x_0})\geq0.
\end{equation}\end{linenomath*}
By using the fact that $w({x_0})\leq w(x)$ for all $x\in \Sigma_\lambda$, we derive that
\begin{linenomath*}\begin{align}\notag
		\mathcal{N}_{\frac{\alpha}{2}}w({x_0})&=C_{n,\alpha}\int_{\mathbb{R}_+^n}\frac{w({x_0})-w(y)}{|{x_0}-y|^{n+{\alpha}}}dy\\
	&=C_{n,\alpha}\bigg\{ \int_{\Sigma_\lambda}\frac{w({x_0})-w(y)}{|{x_0}-y|^{n+{\alpha}}}dy+\int_{\tilde{\Sigma}_\lambda}\frac{w({x_0})-w(y)}{|{x_0}-y|^{n+{\alpha}}}dy\bigg\}\notag\\
	&=C_{n,\alpha}\bigg\{ \int_{\Sigma_\lambda}\frac{w({x_0})-w(y)}{|{x_0}-y|^{n+{\alpha}}}dy+\int_{\Sigma_\lambda}\frac{w({x_0})-w(y^\lambda)}{|{x_0}-y^\lambda|^{n+{\alpha}}}dy\bigg\}\notag\\
	&=C_{n,\alpha}\bigg\{ \int_{\Sigma_\lambda}\frac{w({x_0})-w(y)}{|{x_0}-y|^{n+{\alpha}}}dy+\int_{\Sigma_\lambda}\frac{w({x_0})+w(y)}{|{x_0}-y^\lambda|^{n+{\alpha}}}dy\bigg\}\notag\\
	&\leq C_{n,\alpha} \int_{\Sigma_\lambda}\bigg\{\frac{w({x_0})-w(y)}{|{x_0}-y^\lambda|^{n+{\alpha}}}+\frac{w({x_0})+w(y)}{|{x_0}-y^\lambda|^{n+{\alpha}}}\bigg\}dy\notag\\
	&=C_{n,\alpha}\int_{\Sigma_\lambda}\frac{2w({x_0})}{|{x_0}-y^\lambda|^{n+{\alpha}}}dy\notag\\
	&<0,\label{est:nw2}
\end{align}\end{linenomath*}
which contradicts inequality \eqref{Nx0>0}. Therefore ${x_0} \in \Omega_1$.

By similar computations as in \eqref{compute1} \eqref{compute2} \eqref{compute3} and \eqref{compute4}, we deduce
\begin{linenomath*}\begin{equation}\label{narrow}
(-\Delta)^\frac{\alpha}{2}w({x_0})+a({x_0})w({x_0})\leq w({x_0})\left( 2C_{n,\alpha}\int_{H_\lambda}\frac{1}{|{x_0}-y^\lambda|^{n+{\alpha}}}dy+a({x_0})\right).
\end{equation}\end{linenomath*}
We now estimate the above integral. 
Let $D=\left\{y\in \mathbb{R}^n \mid l<y_1-x_1^0<1,\ |y'-({x_0})'|<1\right\}$, $s=y_1-x_1^0$, $\tau=|y'-({x_0})'|$ and $\omega_{n-2}=|B_1(0)|$ in $\mathbb{R}^{n-2}$. We have 
\begin{linenomath*}\begin{align}\nonumber
	\int_{H_\lambda}\frac{1}{|{x_0}-y^\lambda|^{n+{\alpha}}}dy 
	&=\int_{\tilde{H}_\lambda}\frac{1}{|{x_0}-y|^{n+{\alpha}}}dy\\
	&\geq \int_D\frac{1}{|{x_0}-y|^{n+{\alpha}}}dy \nonumber\\
	&=\int_l^1\int_0^1\frac{\omega_{n-2}\tau^{n-2}d\tau}{(s^2+\tau^2)^{\frac{n+{\alpha}}{2}}}ds \nonumber\\
   &=\int_l^1\int_0^{\frac{1}{s}}\frac{\omega_{n-2}(st)^{n-2}sdt}{s^{n+{\alpha}}(1+t^2)^{\frac{n+{\alpha}}{2}}}ds\nonumber\\
	&=\int_l^1\frac{1}{s^{1+{\alpha}}}\int_0^{\frac{1}{s}}\frac{\omega_{n-2}t^{n-2}dt}{(1+t^2)^{\frac{n+{\alpha}}{2}}}ds \nonumber \\
	&\geq \int_l^1\frac{1}{s^{1+{\alpha}}}\int_0^1\frac{\omega_{n-2}t^{n-2}dt}{(1+t^2)^{\frac{n+{\alpha}}{2}}}ds \nonumber \\
	&\geq C\int_l^1\frac{1}{s^{1+{\alpha}}}ds\geq\frac{c}{l^{\alpha}}  \label{infty}
\end{align}\end{linenomath*}		
where we have used the substitution $\tau=st$.
	
Combining \eqref{narrow} with \eqref{infty}, we arrive at
\begin{linenomath*}\begin{align}
		(-\Delta)^\frac{\alpha}{2}w({x_0})+a({x_0})w({x_0})
		\leq w({x_0})\left( 2C_{n,\alpha}\frac{c}{l^{\alpha}}+a({x_0})\right).
\end{align}\end{linenomath*}
Notice that $a(x)$ is bounded from below in $\Omega_1$ and $w({x_0})<0$. Therefore, when $l$ is sufficiently small, one has
\begin{linenomath*}$$
(-\Delta)^\frac{\alpha}{2}w({x_0})+a({x_0})w({x_0})< 0,
$$ \end{linenomath*}
which contradicts \eqref{lemma2}.  
\end{proof}

Now we give the proof of Theorem \ref{th1} by using the direct method of moving planes.

\begin{proof}[Proof of Theorem \ref{th1}] The proof is divided into three steps.
\begin{step}\label{step1} For sufficiently negative $\lambda$, 
\begin{linenomath*}\begin{equation}
\label{w>0}
w_{\lambda}(x)\geq 0,\quad \forall x \in H_{\lambda}.
\end{equation}\end{linenomath*}
\end{step}

Recall that for each $\lambda \in \mathbb{R}$, $w_\lambda(x)=u_\lambda(x)-u(x)$ defined as in \eqref{def:w}. 
According to the fact that for all $\varphi \in L_\alpha\cap C_{loc}^{1,\theta}(\mathbb{R}^n)$ (with $\theta>\max\{0,\alpha-1\}$),
\begin{linenomath*}\begin{align}
	((-\Delta)^\frac{\alpha}{2}\varphi_\lambda)(x)& =((-\Delta)^\frac{\alpha}{2}\varphi)(x^\lambda), \quad \forall x\in \Sigma_\lambda, \label{phi}\\
	(\mathcal{N}_{\frac{\alpha}{2}}\varphi_\lambda)(x)&=(\mathcal{N}_{\frac{\alpha}{2}}\varphi)(x^\lambda), \quad \forall x\in \{x_n<0\}\cap H_\lambda. \notag
\end{align}\end{linenomath*}
Therefore, setting $\Omega_1=\{x\in H_\lambda \mid x_n\geq 0,\ w_\lambda(x)<0\}$, it follows from \eqref{phi} and the mean value theorem that, for each $x\in \Omega_1$,
\begin{linenomath*}\begin{equation}
	\label{cx}
	(-\Delta)^\frac{\alpha}{2}w_\lambda(x)+a(x)w_\lambda(x)= 0.
\end{equation}\end{linenomath*}
Here, $a(x)=-f'\big(\psi(\lambda, x)\big)$ where $\psi(\lambda, x)$ is a value between $u(x^\lambda)$ and $u(x)$.
Equation \eqref{cx} holds pointwise on $\Omega_1$: for $x\in\Sigma_\lambda^-$ this follows from the equation in $\mathbb{R}^n_+$, while for $x\in H_\lambda\cap\{x_n=0,\ w_\lambda(x)<0\}$ it follows from the fact that, since $u\in C^{1,\theta}_{loc}(\mathbb{R}^n)$, the operator $(-\Delta)^{\alpha/2}u$ is continuous on $\mathbb{R}^n$, so the equation $(-\Delta)^{\alpha/2}u=f(u)$ extends to $\{x_n=0\}$ by continuity.

By \eqref{udecay} and \eqref{fdecay}, for sufficiently negative $\lambda$ and $x\in \Omega_1$, we know that
$u(x)$ is sufficiently small and
$0<u_\lambda(x)<u(x)$. 
Thus,
\begin{linenomath*}$$
a(x)\geq - \psi^{\gamma}(\lambda, x)>-u^{\gamma}(x).
$$ \end{linenomath*}

The decay assumption \eqref{fdecay} instantly yields that
\begin{linenomath*}\begin{align*}
	\label{lim}
	|x|^{\alpha}a(x)> -u^{\gamma}(x)|x|^{\alpha}\geq -\frac{1}{|x|^{\beta\gamma-\alpha}}.
\end{align*}\end{linenomath*}
Hence,
\begin{linenomath*}\begin{align*}
	\varliminf_{|x|\to\infty } |x|^{\alpha}a(x)\geq 0.
\end{align*}\end{linenomath*}

Moreover, for each $x\in H_\lambda\backslash\Sigma_\lambda$, we have $x_1<\lambda$, hence $|x|\geq|x_1|>|\lambda|$. For sufficiently negative $\lambda$, the decay condition \eqref{udecay} ensures $u(x)<d$. Since $w_\lambda(x)<0$ implies $u(x^\lambda)<u(x)<d$, both values lie in $[0,d)$.

Therefore, by using the assumption that $g$ is non-increasing in $[0,d)$, we have
for each $x\in \left\{x\in H_\lambda\backslash \Sigma_\lambda \mid w_\lambda(x)<0\right\}$,
\begin{linenomath*}\begin{equation}
	\label{est:nw}
	\mathcal{N}_{\frac{\alpha}{2}}w_\lambda(x)=(\mathcal{N}_{\frac{\alpha}{2}}u)(x^\lambda)-(\mathcal{N}_{\frac{\alpha}{2}}u)(x)
	=g\big(u(x^\lambda)\big)-g\big(u(x)\big) 
	\geq0.
\end{equation}\end{linenomath*}

Let $\Omega_2=\{x\in H_\lambda \mid x_n<0,\ w_\lambda(x)<0\}$.
By definition $\Omega_2\subset\{x_n<0\}=\mathbb{R}^n\setminus\overline{\mathbb{R}^n_+}$, so $\mathcal{N}_{\alpha/2}w_\lambda$ is well defined on $\Omega_2$. Condition~\eqref{cx} shows $(-\Delta)^{\alpha/2}w_\lambda+aw_\lambda=0\geq 0$ on $\Omega_1$, and \eqref{est:nw} gives $\mathcal{N}_{\alpha/2}w_\lambda\geq 0$ on $\Omega_2$. By construction $w_\lambda\geq 0$ on $H_\lambda\setminus(\Omega_1\cup\Omega_2)$.
It follows from Lemma \ref{Decay} that there exists a constant $R_0>0$, such that if ${x_0}$ is a negative minimum of $w_\lambda(x)$ in $\Sigma_{\lambda}$, then $|{x_0}|\leq R_0.$

Choosing $\lambda<-R_0$, we have
\begin{linenomath*}\begin{align*}
	w_\lambda(x)\geq 0, \quad \forall x \in \Sigma_\lambda.
\end{align*}\end{linenomath*}

\begin{claim}\label{claim0}
\textup{If} $w_\lambda(x)\geq 0,\forall x \in \Sigma_\lambda$, \textup{then} $w_\lambda(x)\geq 0,\forall x \in H_\lambda$.
\end{claim}

Suppose for contradiction that there exists ${x_0} \in H_\lambda \backslash \Sigma_\lambda$ such that $w_\lambda({x_0})<0$. If $x_{0,n}=0$, then $x_0+\frac{1}{k}e_n\in\Sigma_\lambda$ for all $k$, and by continuity $w_\lambda(x_0+\frac{1}{k}e_n)\to w_\lambda(x_0)<0$; but $w_\lambda\geq 0$ on $\Sigma_\lambda$, a contradiction. Hence $x_0$ must satisfy $x_{0,n}<0$, i.e.~$x_0\in\mathbb{R}^n\setminus\overline{\mathbb{R}^n_+}$.  

It follows from \eqref{est:nw} that
\begin{linenomath*}\begin{align}\label{est:nw3}
	\mathcal{N}_{\frac{\alpha}{2}}w_\lambda({x_0})\geq0.
\end{align}\end{linenomath*}
On the other hand, using $w_\lambda({x_0})<0\leq w_\lambda(x), \forall x\in \Sigma_\lambda$ and repeating the computation as in \eqref{est:nw2}, we have
\begin{linenomath*}\begin{align*}
	\mathcal{N}_{\frac{\alpha}{2}}w_\lambda({x_0})
	\leq C_{n,\alpha}\int_{\Sigma_\lambda}\frac{2w_\lambda(x_0)}{|x_0-y^\lambda|^{n+{\alpha}}}dy
	<0,
	\end{align*}\end{linenomath*}
which contradicts \eqref{est:nw3}.

The proof of Step \ref{step1} is completed.

Step \ref{step1} provides a starting point, from which we can move the plane $T_\lambda$ to the right as long as inequality \eqref{w>0} holds. 

Define
\begin{linenomath*}$$
\lambda_0=\sup\left\{\lambda\leq0 \mid w_\mu(x)\geq 0,\forall x\in H_\mu,\forall \mu\leq\lambda\right\}.
$$ \end{linenomath*}
Step \ref{step1} tells us that $\lambda_0>-R_0$. By the continuity of $u$, we have 
\begin{linenomath*}\begin{equation}
	\label{eq:nonegw}
w_{\lambda_0}(x)\geq 0, \quad \forall x \in H_{\lambda_0}.
\end{equation}\end{linenomath*} 
The following claim shows that either $w_{\lambda_0}(x)\equiv 0$ or $w_{\lambda_0}(x)> 0$ in $H_{\lambda_0}$.

\begin{claim}\label{claim1}
If there exists a point $x_0\in H_{\lambda_0}$, such that
$w_{\lambda_0}(x_0)=0,$
then
$w_{\lambda_0}(x)\equiv 0,\forall x \in H_{\lambda_0}.$
\end{claim}
Since we already have \eqref{eq:nonegw}, it suffices to fix a point $x_0\in H_{\lambda_0}$ with $w_{\lambda_0}(x_0)=0$ and consider its position.

If $x_0\in \Sigma_{\lambda_0}$, then $x_0$ is an interior point of $\mathbb{R}^n_+$. On the one hand, we have
\begin{linenomath*}\begin{align*}
	(-\Delta)^\frac{\alpha}{2}w_{\lambda_0}(x_0)&=(-\Delta)^\frac{\alpha}{2}u((x_0)^{\lambda_0})-(-\Delta)^\frac{\alpha}{2}u(x_0)\\
	&=f\big(u((x_0)^{\lambda_0})\big)-f\big(u(x_0)\big)\\
	&=0.
\end{align*}\end{linenomath*}
On the other hand,
\begin{linenomath*}\begin{align}\notag
	(-\Delta)^\frac{\alpha}{2}w_{\lambda_0}(x_0)&=C_{n,\alpha}P.V.\int_{\mathbb{R}^n}\frac{-w_{\lambda_0}(y)}{|x_0-y|^{n+{\alpha}}}dy \\
	&=C_{n,\alpha}P.V.\int_{\Sigma_{\lambda_0}}\frac{-w_{\lambda_0}(y)}{|x_0-y|^{n+{\alpha}}}dy+C_{n,\alpha}P.V.\int_{\tilde{\Sigma}_{\lambda_0}}\frac{-w_{\lambda_0}(y)}{|x_0-y|^{n+{\alpha}}}dy\notag\\
	&+C_{n,\alpha}P.V.\int_{H_{\lambda_0}\backslash \Sigma_{\lambda_0}}\frac{-w_{\lambda_0}(y)}{|x_0-y|^{n+{\alpha}}}dy+C_{n,\alpha}P.V.\int_{\tilde{H}_{\lambda_0}\backslash \tilde{\Sigma}_{\lambda_0}}\frac{-w_{\lambda_0}(y)}{|x_0-y|^{n+{\alpha}}}dy\notag\\
	&=C_{n,\alpha}P.V.\int_{\Sigma_{\lambda_0}}\frac{-w_{\lambda_0}(y)}{|x_0-y|^{n+{\alpha}}}dy+C_{n,\alpha}P.V.\int_{\Sigma_{\lambda_0}}\frac{w_{\lambda_0}(y)}{|x_0-y^{\lambda_0}|^{n+{\alpha}}}dy\notag\\
	&+C_{n,\alpha}P.V.\int_{H_{\lambda_0}\backslash \Sigma_{\lambda_0}}\frac{-w_{\lambda_0}(y)}{|x_0-y|^{n+{\alpha}}}dy+C_{n,\alpha}P.V.\int_{H_{\lambda_0}\backslash \Sigma_{\lambda_0}}\frac{w_{\lambda_0}(y)}{|x_0-y^{\lambda_0}|^{n+{\alpha}}}dy\notag\\
	&=C_{n,\alpha}P.V.\int_{H_{\lambda_0}}\left(\frac{1}{|x_0-y^{\lambda_0}|^{n+{\alpha}}}-\frac{1}{|x_0-y|^{n+{\alpha}}}\right)w_{\lambda_0}(y)dy\leq 0.\label{cl1w<=0}
\end{align}\end{linenomath*}
Here we used that $w_{\lambda_0}(y)\geq 0$ for all $y\in H_{\lambda_0}$ and
\begin{linenomath*}$$
\frac{1}{|x_0-y^{\lambda_0}|} < \frac{1}{|x_0-y|}, \quad \forall y\in H_{\lambda_0}.
$$ \end{linenomath*}
Therefore the above integral must vanish, which forces
\begin{linenomath*}$$w_{\lambda_0}(x)\equiv 0,\quad \forall x \in H_{\lambda_0}.$$ \end{linenomath*}

Now assume that $x_0\in H_{\lambda_0}\backslash \Sigma_{\lambda_0}$. If $x_{0,n}=0$, then $u\in C_{loc}^{1,\theta}(\mathbb{R}^n)$ implies that $(-\Delta)^{\alpha/2}u$ is continuous on $\mathbb{R}^n$, so the equation $(-\Delta)^{\alpha/2}u=f(u)$ extends to $\{x_n=0\}$ by continuity. Hence
\begin{linenomath*}\begin{equation*}
	(-\Delta)^\frac{\alpha}{2}w_{\lambda_0}(x_0)=0.
\end{equation*}\end{linenomath*}
Repeating the computation in the previous case, we again obtain $w_{\lambda_0}(x)\equiv 0$ in $H_{\lambda_0}$.

Finally, assume that $x_{0,n}<0$. Then $x_0\in \mathbb{R}^n\setminus\overline{\mathbb{R}^n_+}$, and the nonlocal Neumann condition yields
\begin{linenomath*}$$
\mathcal{N}_{\frac{\alpha}{2}}w_{\lambda_0}(x_0)=(\mathcal{N}_{\frac{\alpha}{2}}u)(x^{\lambda_0})-(\mathcal{N}_{\frac{\alpha}{2}}u)(x_0)
=g\big(u(x_0^{\lambda_0})\big)-g\big(u(x_0)\big)=0.
$$ \end{linenomath*}
Using the same computation as in \eqref{est:nw2} with $w=w_{\lambda_0}$ and $x_0\in H_{\lambda_0}\backslash\Sigma_{\lambda_0}$, we get
\begin{linenomath*}\begin{align}
	\label{cl1Ns<=0}
	0=\mathcal{N}_{\frac{\alpha}{2}}w_{\lambda_0}({x_0})
	\leq C_{n,\alpha}\int_{\Sigma_{\lambda_0}}\left(\frac{1}{|x_0-y^{\lambda_0}|^{n+{\alpha}}}-\frac{1}{|x_0-y|^{n+{\alpha}}}\right)w_{\lambda_0}(y)dy
  \leq0,
\end{align}\end{linenomath*}
which implies $w_{\lambda_0}(x)\equiv 0$ in $\Sigma_{\lambda_0}$. Since $\Sigma_{\lambda_0}$ is nonempty, we may choose a point $\bar{x}\in \Sigma_{\lambda_0}$; then $w_{\lambda_0}(\bar{x})=0$, and the first case applies to show that $w_{\lambda_0}(x)\equiv 0$ in $H_{\lambda_0}$.
The proof of Claim \ref{claim1} is completed.

We now proceed to prove Theorem \ref{th1}.

\begin{step}\label{step2}
\textup{If} $\lambda_0<0$, \textup{then}
\begin{linenomath*}$$w_{\lambda_0}(x)\equiv 0,\quad \forall x \in H_{\lambda_0}.$$ \end{linenomath*}
\end{step}

By Claim \ref{claim1}, we have 
\begin{linenomath*}$$\text{either } w_{\lambda_0}(x)> 0 \text{ or } w_{\lambda_0}(x)\equiv 0,\quad \forall x \in H_{\lambda_0}.$$ \end{linenomath*}
Assume for contradiction that
\begin{linenomath*}\begin{equation}
	\label{result2}
	w_{\lambda_0}(x)> 0,\quad \forall x\in H_{\lambda_0}.
\end{equation}\end{linenomath*}
It follows from \eqref{result2} that there exists a constant $c_0>0$, such that for a small $\sigma>0$, we have
\begin{linenomath*}$$w_{\lambda_0}(x)\geq c_0,\quad x\in \overline{\Sigma_{\lambda_0 - \sigma}\cap B_{R_0}(0)}.$$ \end{linenomath*}
Since $w_{\lambda}$ depends on $\lambda$ continuously, there exists $\delta>0$ satisfying $\delta+
\lambda_0<0$, such that for all $\lambda \in [\lambda_0,\lambda_0+\delta)$, we have
\begin{linenomath*}\begin{equation}
	\label{eq:pos}
	w_{\lambda}(x)\geq 0,\quad x\in \overline{\Sigma_{\lambda_0 - \sigma}\cap B_{R_0}(0)}.
\end{equation}\end{linenomath*}

For each  $\lambda \in [\lambda_0,\lambda_0+\delta)$, let us
define
\begin{linenomath*}$$\Omega_1=\{x\in H_\lambda \mid x_n\geq 0,\ w_\lambda(x)<0\}\backslash \Sigma_{\lambda_0-\sigma}, \quad \Omega_2=\{x\in H_\lambda \mid x_n<0,\ w_\lambda(x)<0\}.$$ \end{linenomath*}
Since $\Omega_1$ excludes $\Sigma_{\lambda_0-\sigma}=\{x\in\mathbb{R}^n_+\mid x_1<\lambda_0-\sigma\}$ and lies in $H_\lambda=\{x_1<\lambda\}$, we have
\begin{linenomath*}$$\Omega_1\subset \bigl\{x\in\mathbb{R}^n\mid x_n\geq 0,\; \lambda_0-\sigma\leq x_1<\lambda\bigr\} \subset \bigl\{x\in\mathbb{R}^n\mid x_n\geq 0,\; \lambda-(\delta+\sigma)<x_1<\lambda\bigr\},$$ \end{linenomath*}
so $\Omega_1$ is contained in a strip of width $l=\delta+\sigma$ near $T_\lambda$. By choosing $\sigma$ and $\delta$ sufficiently small, the width $l$ satisfies the smallness condition required by Lemma~\ref{Narrow}.
For $x\in \Omega_1\cap \Sigma_\lambda$, the equation in $\mathbb{R}^n_+$ gives $(-\Delta)^{\alpha/2}w_\lambda(x)+a(x)w_\lambda(x)=0$, while for $x\in \Omega_1\cap\{x_n=0\}$ the same identity follows by continuity of $(-\Delta)^{\alpha/2}u$ on $\mathbb{R}^n$. Thus Lemma \ref{Narrow} applies.
According to Lemma \ref{Narrow} and \eqref{eq:pos}, we know
\begin{linenomath*}$$w_{\lambda}(x)\geq 0,\quad \forall x\in H_{\lambda}.$$ \end{linenomath*}

Now we show that the plane can be moved further to the right while preserving inequality \eqref{w>0}, which contradicts the definition of $\lambda_0$.
This completes the proof of Step \ref{step2}.

\begin{step}\label{step3}
The solution $u$ is axially symmetric about a line parallel to the $x_n$-axis.
\end{step}

So far, we have proved that if $\lambda_0<0$, then $u$ is symmetric about the plane $T_{\lambda_0}$.

If $\lambda_0=0$, we move the plane $T_{\lambda}$ from the $+\infty$ to the left. Define
\begin{linenomath*}$$\lambda_0'=\inf\{\lambda \geq0 \mid w_\mu(x)\leq 0,\ \forall x\in H_\mu,\forall \mu\geq\lambda\}.$$ \end{linenomath*}
If $\lambda_0'>0$, through a similar argument as in the previous step, one can show that $w_{\lambda_0'}(x)\equiv 0$, which implies $u$ is symmetric about the plane $T_{\lambda_0'}$. If $\lambda_0'=0=\lambda_0$, then  $w_{\lambda_0}(x)\geq 0, \forall x \in H_0$, and $w_{\lambda_0}(x)\leq 0, \forall x \in H_0$. So we must have $w_{\lambda_0}(x)\equiv 0, \forall x \in H_0$, i.e. $u$ is symmetric about the plane $T_{0}$.

Since the $x_1$ direction can be chosen arbitrarily in $\mathbb{R}^{n-1}$, we conclude that $u$ is axially symmetric about a line parallel to the $x_n$-axis. This completes the proof of the theorem.
\end{proof}

\section{Symmetry results for nonlocal Neumann fractional Laplacian systems} \label{sec:sys} 

In addressing the nonlocal Neumann fractional Laplacian system \eqref{system}, we employ a direct method of moving planes akin to the one utilized in Section \ref{sec:eq}. However, to effectively handle the coupling inherent in the system, we find it imperative to develop a refined version of the decay at infinity and narrow region principles.

We will use $T_\lambda$, $H_\lambda$, $\Sigma_\lambda$ and $x^\lambda$ in the same way as in Section \ref{sec:eq}. 
Define the reflected functions by $u_\lambda(x)=u(x^\lambda)$, $v_\lambda(x)=v(x^\lambda)$ and introduce functions
\begin{linenomath*}$$
U_\lambda(x)=u_\lambda(x)-u(x),\quad V_\lambda(x)=v_\lambda(x)-v(x).
$$ \end{linenomath*}
It follows from \eqref{phi} that for all $x \in \Sigma_\lambda$,
\begin{linenomath*}\begin{align}\label{-delu=f}
	(-\Delta)^\frac{\alpha}{2}U_\lambda(x)
	&=f\big(u(x^\lambda),v(x^\lambda)\big)-f\big(u(x),v(x)\big)\\
	&=\frac{\partial f}{\partial u}(\xi_1(x,\lambda),v(x))U_\lambda(x)+\frac{\partial f}{\partial v}(u_\lambda(x),\eta_1(x,\lambda))V_\lambda(x), \notag 
\end{align}\end{linenomath*}
and
\begin{linenomath*}\begin{align}
	(-\Delta)^\frac{\alpha}{2}V_\lambda(x)=\frac{\partial g}{\partial u}(\xi_2(x,\lambda),v_\lambda(x))U_\lambda(x)+\frac{\partial g}{\partial v}(u(x),\eta_2(x,\lambda))V_\lambda(x)\label{-delv=g},
\end{align}\end{linenomath*}
where $\xi_i(x,\lambda)$ is between $u_\lambda(x)$ and $u(x)$, and $\eta_i(x,\lambda)$ is between $v_\lambda(x)$ and $v(x)$ for $i=1,2.$

Furthermore, we can deduce that for all $x\in \mathbb{R}^n\setminus\overline{\mathbb{R}^n_+}$ with $x\in H_\lambda$,
\begin{linenomath*}\begin{align}\label{nsu=h}
	\mathcal{N}_{\frac{\alpha}{2}}U_\lambda(x)
	&=h\big(u(x^\lambda),v(x^\lambda)\big)-h\big(u(x),v(x)\big) \\
	&=\frac{\partial h}{\partial u}(\xi_3(x,\lambda),v(x))U_\lambda(x)+\frac{\partial h}{\partial v}(u_\lambda(x),\eta_3(x,\lambda))V_\lambda(x),\notag 
\end{align}\end{linenomath*}
and
\begin{linenomath*}\begin{align}
	\mathcal{N}_{\frac{\alpha}{2}}V_\lambda(x)=\frac{\partial l}{\partial u}(\xi_4(x,\lambda),v_\lambda(x))U_\lambda(x)+\frac{\partial l}{\partial v}(u(x),\eta_4(x,\lambda))V_\lambda(x)\label{nsv=l}.
\end{align}\end{linenomath*}
Here, $\xi_i(x,\lambda)$ is between $u_\lambda(x)$ and $u(x)$, and $\eta_i(x,\lambda)$ is between $v_\lambda(x)$ and $v(x)$ for $i=3,4.$

Define 
\begin{linenomath*}$$H_\lambda^{U-}:=\left\{x\in H_\lambda \mid U_\lambda(x)<0\right\},\quad H_\lambda^{V-}:=\left\{x\in H_\lambda \mid V_\lambda(x)<0\right\}.$$ \end{linenomath*}

To prove Theorem \ref{fczth1}, the main ingredients are the following maximum principles, i.e., Decay at infinity and Narrow region principle for systems.
\begin{lemma}(Decay at Infinity Principle for System)
\label{zul1}
	Let $(u,v) \in \left(L_\alpha\cap C_{loc}^{1,\theta}(\mathbb{R}^n)\right)^2$, for some $\theta>\max\{0,\alpha-1\}$, be a positive solution of system \eqref{system} with $f,g,h,l\in C^1([0,+\infty)\times[0,+\infty),\mathbb{R})$. Under assumptions (i)-(vii) as in Theorem \ref{fczth1}, there exists a constant $R_0>0$, for system \eqref{-delu=f}\eqref{-delv=g}\eqref{nsu=h}\eqref{nsv=l} with all $\lambda<0$,\\
(a) if there is $x^*\in {H}_\lambda$, $|x^*|>R_0$ such that $U_\lambda(x^*)=\min_{x\in \overline{H_\lambda}}U_\lambda(x)<0,$
then 
\begin{linenomath*}$$V_\lambda(x^*)<\mu U_\lambda(x^*)<0;$$ \end{linenomath*}
(b) if there is $y^*\in {H}_\lambda$, $|y^*|>R_0$ such that 
$V_\lambda(y^*)=\min_{y\in \overline{H_\lambda}}V_\lambda(y)<0,$
then 
\begin{linenomath*}$$U_\lambda(y^*)<\nu V_\lambda(y^*)<0.$$ \end{linenomath*}
\end{lemma}
We remark that the constants $\mu,\nu>0$ with $\mu\nu>1$ are the constants from assumption (vi) of Theorem~\ref{fczth1}.

\begin{proof}
Let $x^*\in {H}_\lambda$ satisfy
\begin{linenomath*}$$U_\lambda(x^*)=\min_{x\in \overline{H_\lambda}}U_\lambda(x)<0.$$ \end{linenomath*}
There are two possible cases: $x^*_n\geq 0$ and $x^*_n<0$.

\textbf{Case 1.} $x^*_n\geq 0$. Note that when $x^*_n=0$, since $(u,v)\in (C^{1,\theta}_{loc}(\mathbb{R}^n))^2$, the operator $(-\Delta)^{\alpha/2}$ is continuous on $\mathbb{R}^n$, so equation \eqref{-delu=f} extends to $\{x_n=0\}$ by continuity.

Noticing that $x^*$ is the minimum point of $U_\lambda(x)$ in ${H}_\lambda$, we compute
\begin{linenomath*}\begin{align}
		(-\Delta)^\frac{\alpha}{2}U_\lambda(x^*) &=C_{n,\alpha}P.V.\int_{\mathbb{R}^n}\frac{U_\lambda(x^*)-U_\lambda(y)}{|x^*-y|^{n+{\alpha}}}dy \notag\\
		&=C_{n,\alpha}P.V.\left( \int_{H_\lambda}\frac{U_\lambda(x^*)-U_\lambda(y)}{|x^*-y|^{n+{\alpha}}}dy+\int_{\tilde{H}_\lambda}\frac{U_\lambda(x^*)-U_\lambda(y)}{|x^*-y|^{n+{\alpha}}}dy\right)\notag\\
		&=C_{n,\alpha}P.V.\left( \int_{H_\lambda}\frac{U_\lambda(x^*)-U_\lambda(y)}{|x^*-y|^{n+{\alpha}}}dy+\int_{H_\lambda}\frac{U_\lambda(x^*)-U_\lambda(y^\lambda)}{|x^*-y^\lambda|^{n+{\alpha}}}dy\right)\notag\\
		&=C_{n,\alpha}P.V.\left( \int_{H_\lambda}\frac{U_\lambda(x^*)-U_\lambda(y)}{|x^*-y|^{n+{\alpha}}}dy+\int_{H_\lambda}\frac{U_\lambda(x^*)+U_\lambda(y)}{|x^*-y^\lambda|^{n+{\alpha}}}dy\right)\notag\\
		&\leq C_{n,\alpha} \int_{H_\lambda}\left(\frac{U_\lambda(x^*)-U_\lambda(y)}{|x^*-y^\lambda|^{n+{\alpha}}}+\frac{U_\lambda(x^*)+U_\lambda(y)}{|x^*-y^\lambda|^{n+{\alpha}}}\right)dy\notag\\
		&=2C_{n,\alpha}\int_{H_\lambda}\frac{1}{|x^*-y^\lambda|^{n+{\alpha}}}dy\cdot U_\lambda(x^*).\label{zul1compute2}
	\end{align}\end{linenomath*}
Here, we have used the anti-symmetry property of $U_\lambda(x)$.

Combining \eqref{zul1compute2} with \eqref{-delu=f}, we conclude 
\begin{linenomath*}\begin{equation}
\label{f/v<=b}
	\frac{\partial f}{\partial v}(u_\lambda(x^*),\eta_1(x^*,\lambda))V_\lambda(x^*)\leq b_1(x^*,\lambda)U_\lambda(x^*),
\end{equation}\end{linenomath*}
where
\begin{linenomath*}$$b_1(x^*,\lambda)=2C_{n,\alpha}\int_{H_\lambda}\frac{1}{|x^*-y^\lambda|^{n+{\alpha}}}dy-\frac{\partial f}{\partial u}(\xi_1(x^*,\lambda),v(x^*)).$$ \end{linenomath*} 

We claim that $b_1(x^*,\lambda)>0$, for sufficiently large $|x^*|$. In fact, for all $\lambda\leq0$, it follows from \eqref{sec1jifen} that
\begin{linenomath*}\begin{equation}
\label{b_1>=}
b_1(x^*,\lambda)\geq \frac{2C_{n,\alpha}\omega_n}{4^{n+{\alpha}}|x^*|^{\alpha}}-\frac{\partial f}{\partial u}(\xi_1(x^*,\lambda),v(x^*)).
\end{equation}\end{linenomath*}
Owing to assumption (ii), we choose $\delta>0$ such that $\frac{\partial f}{\partial u}(u,v)\leq u^{p-1}v^q$ provided that $u+v<\delta$ and $u,v>0$.
For this particular $\delta$, because of assumption (i), there exists $R_1>0$ such that $0<u(x^*)< {1}/{|x^*|^a},$ $0<v(x^*)<{1}/{|x^*|^b}$ and $u(x^*)+v(x^*)<\delta$, when $|x^*|>R_1$. In addition, since $x^* \in H_\lambda^{U-}$, we have $0<u_\lambda(x^*)<\xi_1(x^*,\lambda)<u(x^*)$. Hence, when $|x^*|>R_1$,
\begin{linenomath*}\begin{align}
	\frac{\partial f}{\partial u}(\xi_1(x^*,\lambda),v(x^*))&\leq \xi_1(x^*,\lambda)^{p-1}v(x^*)^q \notag\\
	&< u(x^*)^{p-1}v(x^*)^q \notag \\
	&\leq \frac{1}{|x^*|^{a(p-1)+qb}}.\label{fu<=}
\end{align}\end{linenomath*}
Putting \eqref{fu<=} into \eqref{b_1>=} and using assumption (vii), we can choose $R_2>R_1$, such that for $|x^*|>R_2$,
\begin{linenomath*}\begin{equation}
\label{b_1}
	b_1(x^*,\lambda)\geq c|x^*|^{-\alpha}>0.
\end{equation}\end{linenomath*}

Thus, combining \eqref{b_1} with assumption (iv), we derive from \eqref{f/v<=b} that
\begin{linenomath*}$$V_\lambda(x^*)\leq \frac{b_1(x^*,\lambda)}{\frac{\partial f}{\partial v}(u_\lambda(x^*),\eta_1(x^*,\lambda))} U_\lambda(x^*)<0,$$ \end{linenomath*}
which implies $x^* \in H_\lambda^{U-}\cap H_\lambda^{V-}$.

Due to assumption (iii), there is $\delta'>0$, such that $\frac{\partial f}{\partial v}(u,v)\leq u^pv^{q-1}$ for $u,v>0$, $u+v<\delta'$. Hence, we can choose $R_3>R_2$ so that for all $|x^*|>R_3$
\begin{linenomath*}$$0<u(x^*)< {1}/{|x^*|^a},\quad 0<v(x^*)<{1}/{|x^*|^b},\quad u(x^*)+v(x^*)<\delta'.$$ \end{linenomath*}
Since we have already shown that $x^* \in H_\lambda^{U-}\cap H_\lambda^{V-}$,  there hold $0<u_\lambda(x^*)<\xi_1(x^*,\lambda)<u(x^*)$, $0<v_\lambda(x^*)<\eta_1(x^*,\lambda)<v(x^*)$.
It follows that for $|x^*|>R_3$
\begin{linenomath*}\begin{align}
	\frac{\partial f}{\partial v}(u_\lambda(x^*),\eta_1(x^*,\lambda))&\leq u_\lambda(x^*)^p\eta_1(x^*,\lambda)^{q-1} \notag\\
	&< u(x^*)^pv(x^*)^{q-1} \notag \\
	&\leq \frac{1}{|x^*|^{ap+qb-b}}.\label{fv<=}
\end{align}\end{linenomath*}
In view of \eqref{b_1} \eqref{fv<=} and assumption (vii), we obtain for $|x^*| > R_3$,
\begin{linenomath*}$$\frac{b_1(x^*,\lambda)}{\frac{\partial f}{\partial v}(u_\lambda(x^*),\eta_1(x^*,\lambda))}>c|x^*|^{ap+b(q-1)-\alpha}\rightarrow \infty,\quad \text{as }|x^*|\rightarrow \infty.$$ \end{linenomath*}
Therefore, there is $R_0>R_3$, for all $|x^*|>R_0$, we have
\begin{linenomath*}$$\frac{b_1(x^*,\lambda)}{\frac{\partial f}{\partial v}(u_\lambda(x^*),\eta_1(x^*,\lambda))}>\mu,$$ \end{linenomath*}
which implies 
\begin{linenomath*}$$V_\lambda(x^*)<\mu U_\lambda(x^*)<0.$$ \end{linenomath*}

We remark that $R_0$ is independent of $\lambda$.

\textbf{Case 2.} $x^*_n<0$, i.e., $x^*\in \mathbb{R}^n\setminus\overline{\mathbb{R}^n_+}$.

A similar calculation as in \eqref{est:nw2} gives us that
\begin{linenomath*}\begin{equation}
	\label{nsU<0}
	\mathcal{N}_{\frac{\alpha}{2}}U_\lambda(x^*)\leq 2C_{n,\alpha}\int_{\Sigma_\lambda}\frac{1}{|x^*-y^\lambda|^{n+{\alpha}}}dy\cdot U_\lambda(x^*) <0.
\end{equation}\end{linenomath*}

Combining the above estimate with \eqref{nsu=h}, we find
\begin{linenomath*}\begin{equation}
\label{v<u}
\frac{\partial h}{\partial v}(u_\lambda(x^*),\eta_3(x^*,\lambda))V_\lambda(x^*)<-\frac{\partial h}{\partial u}(\xi_3(x^*,\lambda),v(x^*))U_\lambda(x^*).
\end{equation}\end{linenomath*}
It follows from assumptions (v)(vi) and  \eqref{v<u} that 
\begin{linenomath*}\begin{equation}
\label{conclusion}
V_\lambda(x^*)<\frac{-\frac{\partial h}{\partial u}(\xi_3(x^*,\lambda),v(x^*))}{\frac{\partial h}{\partial v}(u_\lambda(x^*),\eta_3(x^*,\lambda))}U_\lambda(x^*)\leq \mu U_\lambda(x^*)<0.
\end{equation}\end{linenomath*}
Thus, the proof of conclusion (a) is completed.

The proof of conclusion (b) is similar and we shall omit it.
\end{proof}

\begin{lemma}(Narrow Region Principle For System) \label{zul2}
	Let $(u,v) \in \left(L_\alpha\cap C_{loc}^{1,\theta}(\mathbb{R}^n)\right)^2$, for some $\theta>\max\{0,\alpha-1\}$, be a positive solution of system \eqref{system} with $f,g,h,l\in C^1([0,+\infty)\times[0,+\infty),\mathbb{R})$. Assume 
\begin{linenomath*}$$(i')\lim_{|x|\to\infty } u(x)= 0,\quad \lim_{|x|\to\infty } v(x)= 0,$$ \end{linenomath*}
and (iv)-(vi) as in Theorem \ref{fczth1}. Then there exists $l_0>0$ such that for each $l\in(0, l_0]$, $U_\lambda(x)$ and $V_\lambda(x)$ satisfy:\\
(c) if there is $x^*\in \Omega_{\lambda,l}:=\{x\in {H}_\lambda \mid \lambda-l<x_1<\lambda\}$ such that $U_\lambda(x^*)=\min_{x\in \overline{H_\lambda}}U_\lambda(x)<0,$
then 
\begin{linenomath*}$$V_\lambda(x^*)<\mu U_\lambda(x^*)<0;$$ \end{linenomath*}
(d) if there is $y^*\in \Omega_{\lambda,l}$ such that $V_\lambda(y^*)=\min_{y\in \overline{H_\lambda}}V_\lambda(y)<0,$ 
then 
\begin{linenomath*}$$U_\lambda(y^*)<\nu V_\lambda(y^*)<0.$$ \end{linenomath*}
\end{lemma}
\begin{proof}
Without loss of generality, let 
\begin{linenomath*}$$x^*\in \Omega_{\lambda,l}\quad \text{and} \quad U_\lambda(x^*)=\min_{x\in \overline{H_\lambda}}U_\lambda(x)<0,$$ \end{linenomath*}
for some $l>0$.

If $x^*_n<0$, we establish the validity of statement \textit{(c)}, by repeating the arguments presented in \textbf{Case 2.} of the proof for Lemma \ref{zul1}.

We now examine the scenario where $x^*_n\geq 0$. When $x^*_n=0$, equation \eqref{-delu=f} extends to $\{x_n=0\}$ by continuity of $(-\Delta)^{\alpha/2}$, so the following argument applies uniformly for all $x^*_n\geq 0$.
Applying the identical computation as presented in \eqref{zul1compute2} and \eqref{f/v<=b}, we deduce that
\begin{linenomath*}\begin{equation}
\label{f/v<=b1u}
	\frac{\partial f}{\partial v}(u_\lambda(x^*),\eta_1(x^*,\lambda))V_\lambda(x^*)\leq b_1(x^*,\lambda)U_\lambda(x^*),
\end{equation}\end{linenomath*}
where
\begin{linenomath*}$$b_1(x^*,\lambda)=2C_{n,\alpha}\int_{H_\lambda}\frac{1}{|x^*-y^\lambda|^{n+{\alpha}}}dy-\frac{\partial f}{\partial u}(\xi_1(x^*,\lambda),v(x^*)).$$ \end{linenomath*}

We assert that for sufficiently small $l$, $b_1(x^*,\lambda)>0$. By a similar argument as in \eqref{infty}, we deduce that as $l\rightarrow0^+$, 
\begin{linenomath*}$$
\int_{H_\lambda}\frac{1}{|x^*-y^\lambda|^{n+{\alpha}}}dy\rightarrow +\infty.
$$ \end{linenomath*}
Simultaneously, thanks to assumption $(i')$, $u(x)$ and $v(x)$ are bounded functions on $\mathbb{R}^n$, implying that $\xi_1(x,\lambda)$ is also bounded. Due to $f\in C^1([0,+\infty)\times[0,+\infty),\mathbb{R})$, there exists some $c>0$ such that 
\begin{linenomath*}$$\left|\frac{\partial f}{\partial u}(\xi_1(x^*,\lambda),v(x^*))\right|<c,\quad \text{for all } \lambda.$$ \end{linenomath*}
Therefore, there exists $l_1>0$ such that for all $0<l\leq l_1$, we have $b_1(x^*,\lambda)>0$. In fact, we obtain
\begin{linenomath*}\begin{equation}
\label{b1toinfty}
  \lim_{l\to 0^+}b_1(x^*,\lambda)=+\infty.
\end{equation}\end{linenomath*}
Combining this estimate with assumption (iv), we deduce from \eqref{f/v<=b1u} that for all $0<l<l_1$,
\begin{linenomath*}\begin{equation}
\label{V<=U}
V_\lambda(x^*)\leq \frac{b_1(x^*,\lambda)}{\frac{\partial f}{\partial v}(u_\lambda(x^*),\eta_1(x^*,\lambda))} U_\lambda(x^*)<0.
\end{equation}\end{linenomath*}

Note that $\frac{\partial f}{\partial v}(u_\lambda(x^*),\eta_1(x^*,\lambda))$ is uniformly bounded with respect to $\lambda$, i.e. there is $c>0$ such that
\begin{linenomath*}\begin{equation}
\label{f/v<c}
  0<\frac{\partial f}{\partial v}(u_\lambda(x^*),\eta_1(x^*,\lambda))<c,\quad \forall \lambda \in \mathbb{R}.
\end{equation}\end{linenomath*}
It follows from \eqref{b1toinfty}–\eqref{f/v<c} that there is $l_0\in (0,l_1)$ such that for all $l\in (0,l_0]$,
\begin{linenomath*}\begin{equation}
\label{b1>mu}
	\frac{b_1(x^*,\lambda)}{\frac{\partial f}{\partial v}(u_\lambda(x^*),\eta_1(x^*,\lambda))}>\mu.
\end{equation}\end{linenomath*}
Substituting \eqref{b1>mu} into \eqref{V<=U}, we obtain
\begin{linenomath*}$$V_\lambda(x^*)<\mu U_\lambda(x^*)<0.$$ \end{linenomath*}
This implies conclusion (c).

The proof of conclusion (d) is similar and we shall omit it.
\end{proof}

\begin{proof}[Proof of Theorem \ref{fczth1}]
	In the following we give the proof of  Theorem \ref{fczth1} through a three-step application of the direct method of moving planes procedure.
\setcounter{step}{0}
\begin{step}
\label{fczs1}
\textup{There exists $\lambda^*<0$ such that}
\begin{linenomath*}\begin{equation}
\label{u,v>0}
  U_{\lambda}(x)\geq 0\quad \textup{and} \quad V_{\lambda}(x)\geq 0,\quad \forall x \in H_{\lambda}.
\end{equation}\end{linenomath*}
\textup{for all} $\lambda\leq\lambda^*$.
\end{step}
Choose $\lambda^*<-R_0$, where $R_0$ is given by Lemma \ref{zul1}. We show that $U_{\lambda}(x)\geq 0$ and $V_{\lambda}(x)\geq 0$ in $H_{\lambda}$ for all $\lambda\leq \lambda^* $.

Assume for contradiction that there is a $\lambda \leq \lambda^*$ and a point $x^* \in H_{\lambda}$ such that $U_{\lambda}(x^*)<0$. Without loss of generality, we assume
\begin{linenomath*}$$U_\lambda(x^*)=\min_{x\in \overline{H_\lambda}}U_\lambda(x)<0.$$ \end{linenomath*}
Since $\lambda \leq \lambda^* < -R_0$, we know that $|x^*|>R_0$. An immediate consequence of Lemma \ref{zul1} is
\begin{linenomath*}\begin{equation}
\label{pro2v<0}
	V_\lambda(x^*)<\mu U_\lambda(x^*)<0.
\end{equation}\end{linenomath*}

From assumption (i), we know that $\lim_{|x|\to\infty } V_\lambda(x)= 0$. Moreover, $V_\lambda(x)= 0$, for $x \in T_\lambda$. Hence, there exists a point $y^* \in H_{\lambda}$ such that
\begin{linenomath*}$$V_\lambda(y^*)=\min_{y\in \overline{H_\lambda}}V_\lambda(y)<0.$$ \end{linenomath*}
In view of $|y^*|>R_0$, it follows from Lemma \ref{zul1} that
\begin{linenomath*}\begin{equation}
\label{pro2u<0}
	U_\lambda(y^*)<\nu V_\lambda(y^*)<0.
\end{equation}\end{linenomath*}
Combining \eqref{pro2v<0} and \eqref{pro2u<0}, we obtain
\begin{linenomath*}$$V_\lambda(x^*)<\mu U_\lambda(x^*) \leq \mu U_\lambda(y^*)<\mu\nu V_\lambda(y^*) \leq \mu\nu V_\lambda(x^*).$$ \end{linenomath*}
Noticing that $V_\lambda(x^*)<0$ and $\mu\nu>1$, we obtain a contradiction.
Hence for all $\lambda\leq\lambda^*$, we must have
\begin{linenomath*}$$U_{\lambda}(x)\geq 0, \quad V_{\lambda}(x)\geq 0, \quad \forall x \in H_{\lambda}.$$ \end{linenomath*}
This completes the proof of Step \ref{fczs1}.

We now move the hyperplane $T_\lambda$ to the right as long as \eqref{u,v>0} holds to its limiting position. Define
\begin{linenomath*}$$
\lambda_0=\sup\left\{\lambda\leq0 \mid U_\mu(x)\geq 0, V_\mu(x)\geq 0, \forall x\in H_\mu, \forall \mu\leq\lambda\right\}.
$$ \end{linenomath*}

Step \ref{fczs1} indicates that $\lambda_0>-\infty$. Obviously, $\lambda_0 \leq0$. Since all the functions we consider are continuous with respect to $\lambda$, we know that $U_{\lambda_0}(x)\geq 0$ and $V_{\lambda_0}(x)\geq 0$, for all $x\in H_{\lambda_0}$.

Before proceeding further, we shall investigate the properties of functions $U_{\lambda_0}(x)$ and $V_{\lambda_0}(x)$.

\setcounter{claim}{0}

\begin{claim}\label{fczc1}
If $U_{\lambda_0}(x)\equiv0$ in $H_{\lambda_0}$, then $V_{\lambda_0}(x)\equiv 0$ in $H_{\lambda_0}$. If $V_{\lambda_0}(x)\equiv0$ in $H_{\lambda_0}$, then $U_{\lambda_0}(x)\equiv 0$ in $H_{\lambda_0}$.
\end{claim}
\begin{proof}
If $U_{\lambda_0}(x)\equiv0$ in $H_{\lambda_0}$, then for every $x\in \Sigma_{\lambda_0}$, equation \eqref{-delu=f} yields
\begin{linenomath*}$$0=(-\Delta)^\frac{\alpha}{2}U_{\lambda_0}(x)=\frac{\partial f}{\partial v}(u_{\lambda_0}(x),\eta_1(x,\lambda_0))V_{\lambda_0}(x).$$ \end{linenomath*}
For points $x\in H_{\lambda_0}\cap\{x_n=0\}$, since $(u,v)\in (C_{loc}^{1,\theta}(\mathbb{R}^n))^2$, the operator $(-\Delta)^{\alpha/2}$ is continuous on $\mathbb{R}^n$, so \eqref{-delu=f} extends to $\{x_n=0\}$ by continuity and the same identity remains valid. For points $x\in H_{\lambda_0}\cap\{x_n<0\}=\mathbb{R}^n\setminus\overline{\mathbb{R}^n_+}$, equation \eqref{nsu=h} gives
\begin{linenomath*}$$0=\mathcal{N}_{\frac{\alpha}{2}}U_{\lambda_0}(x)=\frac{\partial h}{\partial v}(u_{\lambda_0}(x),\eta_3(x,{\lambda_0}))V_{\lambda_0}(x).$$ \end{linenomath*}
Therefore, by assumptions (iv) and (v), we obtain $V_{\lambda_0}(x)\equiv 0$ in $H_{\lambda_0}$. The proof of the second conclusion is analogous, using \eqref{-delv=g} in $H_{\lambda_0}\cap\{x_n\geq 0\}$ and \eqref{nsv=l} in $H_{\lambda_0}\cap\{x_n<0\}$.
\end{proof}

\begin{claim}\label{fczc2}
If $U_{\lambda_0}(x)\not\equiv0$ or $V_{\lambda_0}(x)\not\equiv0$ for $x\in H_{\lambda_0}$, then $U_{\lambda_0}(x)>0$ and $V_{\lambda_0}(x)>0$, for all $x\in H_{\lambda_0}$.
\end{claim}

\begin{proof}
If $U_{\lambda_0}(x)\not\equiv0$ in $H_{\lambda_0}$ and we already know $U_{\lambda_0}(x)\geq0$ for all $x\in H_{\lambda_0}$, due to the continuity of $U_\lambda(x)$ with respect to $\lambda$, our aim is to further demonstrate that $U_{\lambda_0}(x)>0$, $\forall x\in H_{\lambda_0}$. To prove this, we assume for contradiction that there exists a point $x^*\in H_{\lambda_0}$ such that
\begin{linenomath*}$$
U_{\lambda_0}(x^*)=0.
$$ \end{linenomath*}

If $x^*_n\geq 0$, since $(u,v)\in (C^{1,\theta}_{loc}(\mathbb{R}^n))^2$, equation \eqref{-delu=f} holds at $x^*$ (extending to $x^*_n=0$ by continuity of $(-\Delta)^{\alpha/2}$). Employing the same computation as in \eqref{cl1w<=0}, we can establish
\begin{linenomath*}\begin{align}\label{eq:Du-neg}
	(-\Delta)^\frac{\alpha}{2}U_{\lambda_0}(x^*)=C_{n,\alpha}P.V \int_{H_{\lambda_0}}\left(\frac{1}{|x^*-y^{\lambda_0}|^{n+{\alpha}}}-\frac{1}{|x^*-y|^{n+{\alpha}}}\right)U_{\lambda_0}(y)dy<0.
\end{align}\end{linenomath*}
While it is deduced from \eqref{-delu=f} and assumption \textit{(iv)} that
\begin{linenomath*}$$(-\Delta)^\frac{\alpha}{2}U_{\lambda_0}(x^*)=\frac{\partial f}{\partial v}(u_{\lambda_0}(x^*),\eta_1(x^*,\lambda_0))V_{\lambda_0}(x^*)\geq0,$$ \end{linenomath*}
this contradicts with \eqref{eq:Du-neg}.

If $x^*_n<0$, i.e., $x^*\in\mathbb{R}^n\setminus\overline{\mathbb{R}^n_+}$, by using \eqref{cl1Ns<=0}, we immediately deduce
\begin{linenomath*}\begin{align*}
	\mathcal{N}_{\frac{\alpha}{2}}U_{\lambda_0}(x^*)=C_{n,\alpha} \int_{\Sigma_{\lambda_0}}\left(\frac{1}{|x^*-y^{\lambda_0}|^{n+{\alpha}}}-\frac{1}{|x^*-y|^{n+{\alpha}}}\right)U_{\lambda_0}(y)dy<0.
\end{align*}\end{linenomath*}
However, according to \eqref{nsu=h} and assumption \textit{(v)}, we know that 
\begin{linenomath*}$$\mathcal{N}_{\frac{\alpha}{2}}U_\lambda(x^*)=\frac{\partial h}{\partial v}(u_{\lambda_0}(x^*),\eta_3(x^*,{\lambda_0}))V_{\lambda_0}(x^*)\geq0.$$ \end{linenomath*}
This is a contradiction. 

As a result, we establish $U_{\lambda_0}(x)>0$, for all $x\in H_{\lambda_0}$. Given this situation, Claim \ref{fczc1} implies that $V_{\lambda_0}(x)\not\equiv0$. By employing a similar argument, it can be demonstrated that $V_{\lambda_0}(x)>0$, for all $x\in H_{\lambda_0}$.
\end{proof}

We now move forward to prove Theorem \ref{fczth1}.

\begin{step}
\label{fczs2}
\textup{If} $\lambda_0<0$, \textup{then}
\begin{linenomath*}$$U_{\lambda_0}(x)\equiv0,\quad V_{\lambda_0}(x)\equiv0,\quad \forall x\in H_{\lambda_0}.$$ \end{linenomath*}
\end{step}

Referring to Claim \ref{fczc2}, our goal is to eliminate the possibility where both $U_{\lambda_0}$ and $V_{\lambda_0}$ are strictly positive in $H_{\lambda_0}$. Let us suppose, for the sake of argument, that this scenario holds:
\begin{linenomath*}\begin{equation}
\label{thu>0,v>0}
U_{\lambda_0}(x)>0,\quad V_{\lambda_0}(x)>0,\quad \forall x\in H_{\lambda_0}.
\end{equation}\end{linenomath*}

Considering the definition of $\lambda_0$ (which is assumed to be less than $0$), we can identify sequences
$\{\lambda_k\}^\infty_{k=1}$ and $\{x^k\}^\infty_{k=1}$ such that
\begin{linenomath*}\begin{equation}
\label{sequence}
\lambda_0<\lambda_{k+1}<\lambda_k<0,\quad k=1,2,\ldots; \quad \lim_{k\to\infty}\lambda_k=\lambda_0;
\end{equation}\end{linenomath*}
$x^k\in H_{\lambda_k}$, and either $U_{\lambda_k}(x^k)<0$ or $V_{\lambda_k}(x^k)<0$. Without loss of generality, let us assume $U_{\lambda_k}(x^k)<0$ (up to a sub-sequence) and that $x^k$ corresponds to the minimum points, i.e.
\begin{linenomath*}\begin{equation}
\label{minuxk}
U_{\lambda_k}(x^k)=\min_{x\in \overline{H_k}}U_{\lambda_k}(x)<0,\quad k=1,2,\ldots.
\end{equation}\end{linenomath*}

There are two possible cases.

\begin{case}
The sequence $\{x^k\}^\infty_{k=1}$ contains a bounded sub-sequence.
\end{case}

Without loss of generality, we assume
\begin{linenomath*}\begin{equation}
\label{limxk}
\lim_{k\to\infty}x^k=x^*.
\end{equation}\end{linenomath*}
From \eqref{thu>0,v>0} \eqref{sequence} and \eqref{minuxk}, it is evident that
\begin{linenomath*}$$x^*\in \cap_{k=1}^{+\infty}H_{\lambda_0}=\overline{H_{\lambda_0}},\quad U_{\lambda_0}(x^*)=0.$$ \end{linenomath*}
Given that $U_{\lambda_0}(x)>0$ for all $x\in H_{\lambda_0}$, it implies that $x^*$ must lie on the boundary of $ H_{\lambda_0}$, i.e. $x^*\in T_{\lambda_0}$. Taking \eqref{sequence} and \eqref{limxk} into account, for $l_0>0$ defined by Lemma \ref{zul2}, we can pick $K_1>0$ such that if $k>K_1$, then $\lambda_0<\lambda_k<\lambda_0+l_0/2<0$ and
\begin{linenomath*}$$
x^k\in \Omega_{\lambda_k+l_0/2,l_0}.
$$ \end{linenomath*}
Here $l_0$ is the strip width threshold provided by Lemma~\ref{zul2}, such that conclusions (c)--(d) hold for all $l\in(0,l_0]$.

According to Lemma \ref{zul2}, we have
\begin{linenomath*}\begin{equation}
\label{step2v<u}
V_{\lambda_k}(x^k)<\mu U_{\lambda_k}(x^k)<0.
\end{equation}\end{linenomath*}

It follows that there is $y^k\in H_{\lambda_k}$, such that
\begin{linenomath*}$$V_{\lambda_k}(y^k)=\min_{y\in \overline{H_{\lambda_k}}}V_{\lambda_k}(y)<0,\quad k=K_1+1,K_1+2,\ldots.$$ \end{linenomath*}
For sequence $\{y^k\}^\infty_{K_1+1}$, there are also two possible cases.

\begin{subcase}\label{case:bb}
$\{y^k\}^\infty_{K_1+1}$ has a bounded sub-sequence.
\end{subcase}

In this case, we also denote the convergent sub-sequence as $y^k$, i.e. $\lim_{k\to\infty}y^k=y^*$. Consequently, we have
$V_{\lambda_0}(y^*)= 0$, implying that $y^*\in T_{\lambda_0}$. Therefore, there is $K>K_1$, such that for $k>K$, $y^k\in \Omega_{\lambda_k+l_0/2,l_0}$. Applying Lemma \ref{zul2},
we obtain
\begin{linenomath*}\begin{equation}
\label{step2u<v}
U_{\lambda_k}(y^k)<\nu V_{\lambda_k}(y^k)<0.
\end{equation}\end{linenomath*}

By combining \eqref{step2u<v} with \eqref{step2v<u}, for $k>K$,  we get 
\begin{linenomath*}$$U_{\lambda_k}(y^k)<\nu V_{\lambda_k}(y^k)\leq \nu V_{\lambda_k}(x^k)<\mu\nu U_{\lambda_k}(x^k)\leq \mu\nu U_{\lambda_k}(y^k),$$ \end{linenomath*}
resulting in a contradiction, given that $U_{\lambda_k}(y^k)<0$.

\begin{subcase}\label{case:bub}
$\lim_{k\to\infty}|y^k|=\infty$.
\end{subcase}

In this case, there is $K>K_1$, such that for all $k>K$, $|y^k|>R_0$, where $R_0$ is chosen as in Lemma \ref{zul1}. It follows from Lemma \ref{zul1} that
\begin{linenomath*}\begin{equation}
\label{case1.2}
U_{\lambda_k}(y^k)<\nu V_{\lambda_k}(y^k)<0,\quad \forall k>K.
\end{equation}\end{linenomath*}

By combining \eqref{case1.2} with \eqref{step2v<u}, for $k>K$, we find 
\begin{linenomath*}$$U_{\lambda_k}(y^k)<\nu V_{\lambda_k}(y^k)\leq \nu V_{\lambda_k}(x^k)<\mu\nu U_{\lambda_k}(x^k)\leq \mu\nu U_{\lambda_k}(y^k),$$ \end{linenomath*}
which leads to a contradiction, as $U_{\lambda_k}(y^k)<0$.

\begin{case}
$\lim_{k\to\infty}|x^k|=\infty$.
\end{case}

In this situation, we can find $K_2>0$ such that for $k>K_2$, it holds that $|x^k|>R_0$, where $R_0$ is chosen as in Lemma \ref{zul1}.  It follows from Lemma \ref{zul1} that
\begin{linenomath*}\begin{equation}
\label{case2}
V_{\lambda_k}(x^k)<\mu U_{\lambda_k}(x^k)<0,\quad \forall k>K_2.
\end{equation}\end{linenomath*}
Hence, for each $k>K_2$, there is $y^k\in H_{\lambda_k}$
satisfying
\begin{linenomath*}$$V_{\lambda_k}(y^k)=\min_{y\in \overline{H_{\lambda_k}}}V_{\lambda_k}(y)<0,\quad \forall k>K_2.$$ \end{linenomath*}

There are two possible cases.

\begin{subcase}
$\{y^k\}^\infty_{k=1}$ has a bounded sub-sequence.
\end{subcase}

By a similar argument as in Case \ref{case:bub}, we will find a contradiction in this case.

\begin{subcase}
$\lim_{k\to\infty}|y^k|=\infty$.
\end{subcase}

Choosing $K>K_2$ such that for $k>K$, $|y^k|>R_0$, where $R_0$ is determined by Lemma \ref{zul1}. Applying Lemma \ref{zul1}, we then have
\begin{linenomath*}\begin{equation}
\label{case2.2}
U_{\lambda_k}(y^k)<\nu V_{\lambda_k}(y^k)<0,\quad \forall k>K.
\end{equation}\end{linenomath*}

Hence, combining \eqref{case2} and \eqref{case2.2}, when $k>K$, we have 
\begin{linenomath*}$$U_{\lambda_k}(y^k)<\nu V_{\lambda_k}(y^k)\leq \nu V_{\lambda_k}(x^k)<\mu\nu U_{\lambda_k}(x^k)\leq \mu\nu U_{\lambda_k}(y^k),$$ \end{linenomath*}
resulting in a contradiction, as $U_{\lambda_k}(y^k)<0$.

In conclusion, if $\lambda_0<0$, we have
\begin{linenomath*}$$U_{\lambda_0}(x)\equiv0,\quad V_{\lambda_0}(x)\equiv0,\quad \forall x\in H_{\lambda_0}.$$ \end{linenomath*}

\begin{step}
The solution $(u,v)$ is axially symmetric about a line parallel to the $x_n$-axis.
\end{step}

First, by a similar argument as in Step \ref{step3} of the previous section, we deduce that $(u,v)$ is symmetric with respect to some hyperplane $\{x \in \mathbb{R}^n \mid x_1=c\}$. Since the $x_1$ direction can be chosen arbitrarily, we have actually shown that $(u,v)$ is axially symmetric about a line parallel to the $x_n$-axis. 
\end{proof}


\end{document}